\documentclass[12pt]{amsart}
\usepackage{amscd,amsmath,amsthm,amssymb}
\usepackage[left]{lineno}
\usepackage{color}
\usepackage{stmaryrd}
\usepackage[utf8]{inputenc}
\usepackage{cleveref}
\usepackage{epstopdf}
\usepackage{graphicx}
\usepackage{xcolor}

\definecolor{verylight}{gray}{0.97}
\definecolor{light}{gray}{0.9}
\definecolor{medium}{gray}{0.85}
\definecolor{dark}{gray}{0.6}

 \def\frk{\mathfrak}               

 \def\mm{{\frk m}}

 \def\G{{\mathcal G}}

 \def\0b{{\mathbf 0}}

 \def\opn#1#2{\def#1{\operatorname{#2}}} 
 \opn\chara{char} \opn\length{\ell} \opn\pd{pd} \opn\rk{rk}
 \opn\projdim{proj\,dim} \opn\injdim{inj\,dim} \opn\rank{rank}
 \opn\depth{depth} \opn\grade{grade} \opn\height{height}
 \opn\embdim{emb\,dim} \opn\codim{codim}
 
 \opn\Tr{Tr} \opn\bigrank{big\,rank}
 \opn\superheight{superheight}\opn\lcm{lcm}
 \opn\trdeg{tr\,deg}
 \opn\reg{reg} \opn\lreg{lreg} \opn\ini{in} \opn\lpd{lpd}
 \opn\size{size} \opn\sdepth{sdepth}
 \opn\link{link}\opn\fdepth{fdepth}\opn\lex{lex}
 \opn\tr{tr}
 \opn\type{type}
 \opn\gap{gap}
 \opn\arithdeg{arith-deg}
 \opn\HS{HS}
 \opn\GL{GL}
 \opn\div{div} \opn\Div{Div} \opn\cl{cl} \opn\Cl{Cl}
 \opn\Spec{Spec} \opn\Supp{Supp} \opn\supp{supp} \opn\Sing{Sing}
 \opn\Ass{Ass} \opn\Min{Min}\opn\Mon{Mon}
 \opn\Ann{Ann} \opn\Rad{Rad} \opn\Soc{Soc}\opn\Deg{Deg}
 \opn\Im{Im} \opn\Ker{Ker} \opn\Coker{Coker} \opn\Am{Am}
 \opn\Hom{Hom} \opn\Tor{Tor} \opn\Ext{Ext} \opn\End{End}
 \opn\Aut{Aut} \opn\id{id}
 
 \opn\nat{nat}
 \opn\pff{pf}
 \opn\Pf{Pf} \opn\GL{GL} \opn\SL{SL} \opn\mod{mod} \opn\ord{ord}
 \opn\Gin{Gin} \opn\Hilb{Hilb}\opn\sort{sort}
 \opn\PF{PF}\opn\Ap{Ap}
 \opn\mult{mult}
 \opn\bight{bight}
 \opn\aff{aff}
 \opn\relint{relint} \opn\st{st}
 \opn\lk{lk} \opn\cn{cn} \opn\core{core} \opn\vol{vol}  \opn\inp{inp} \opn\nilpot{nilpot}
 \opn\link{link} \opn\star{star}\opn\lex{lex}\opn\set{set}
 \opn\width{wd}
 \opn\Fr{F}
 \opn\QF{QF}
 \opn\G{G}
 \opn\type{type}\opn\res{res}
 \opn\conv{conv}
 \opn\Ind{Ind}
 \opn\gr{gr}
 
 \def\pot#1#2{#1[\kern-0.28ex[#2]\kern-0.28ex]}

 \opn\dirlim{\underrightarrow{\lim}}
 \opn\inivlim{\underleftarrow{\lim}}
 \let\dirsum=\oplus

 \let\to=\rightarrow
 
 \def\Implies{\ifmmode\Longrightarrow \else
         \unskip${}\Longrightarrow{}$\ignorespaces\fi}
 \def\implies{\ifmmode\Rightarrow \else
         \unskip${}\Rightarrow{}$\ignorespaces\fi}
 \def\iff{\ifmmode\Longleftrightarrow \else
         \unskip${}\Longleftrightarrow{}$\ignorespaces\fi}

 \let\:=\colon
 \newtheorem{Theorem}{Theorem}[section]
 \newtheorem{Lemma}[Theorem]{Lemma}
 \newtheorem{Corollary}[Theorem]{Corollary}
 
 \newtheorem{Remark}[Theorem]{Remark}
 
 \newtheorem{Example}[Theorem]{Example}
 
 \newtheorem{Definition}[Theorem]{Definition}

 \let\epsilon\varepsilon
 \let\kappa=\varkappa
 \def\qed{\ifhmode\textqed\fi
       \ifmmode\ifinner\quad\qedsymbol\else\dispqed\fi\fi}
 \def\textqed{\unskip\nobreak\penalty50
        \hskip2em\hbox{}\nobreak\hfil\qedsymbol
        \parfillskip=0pt \finalhyphendemerits=0}
 \def\dispqed{\rlap{\qquad\qedsymbol}}
 
 \opn\dis{dis}
 \def\pnt{{\raise0.5mm\hbox{\large\bf.}}}
 
 \opn\Lex{Lex}

\begin{document}

\title{Freiman Borel type ideals}

\author{Guangjun Zhu$^{\ast}$, Yakun Zhao, Shiya Duan and  Yulong Yang  }


\address{ School of Mathematical Sciences, Soochow University, Suzhou, Jiangsu, 215006, P. R. China}

\email{zhuguangjun@suda.edu.cn(Corresponding author:Guangjun Zhu),
\linebreak[4]1768868280@qq.com(Yakun Zhao),3136566920@qq.com(Shiya Duan), 1975992862@qq.com
\linebreak(Yulong Yang).}

\thanks{$^{\ast}$ Corresponding author}





\maketitle
\begin{abstract}
An equigenerated monomial ideal $I$ in the polynomial ring  $S= K[x_1,\ldots,x_n]$ is a
Freiman ideal if $\mu(I^2)=\ell(I)\mu(I)-{\ell(I)\choose 2}$ where $\ell(I)$ is the analytic spread of $I$ and $\mu(I)$ is the number of minimal generators of $I$.  In this paper, we classify
 certain classes of Borel type ideals, including Borel ideals with multiple Borel generators  and principal $k$-Borel ideals,  which are Freiman.
\end{abstract}

\section*{Introduction}

The concept  of Freiman ideals appeared the first time in \cite{HZ1}. Based on a famous theorem of Freiman \cite{Fre}, it was shown in \cite{HMZ} that if $I$ is an
equigenerated monomial ideal, i.e., all its
generators are of the same degree, then $\mu(I^2)\geq \ell(I)\mu(I)-{\ell(I)\choose 2}$. Here $\mu(I)$ denotes the minimal number of generators of the ideal $I$  and $\ell(I)$  denotes its analytic spread which by definition is the Krull dimension of the
fiber ring $F(I)=\dirsum_{k\geq 0}I^k/\mm I^k$, where $\mm$ denotes the graded maximal ideal of $S$.
If the equality holds, then Herzog and Zhu in \cite{HZ1}  called the ideal $I$ a Freiman ideal (or simply Freiman).
As a generalization of Freiman’s theorem, a result was proved by B\"or\"oczky et. al.
\cite{BSS} which in algebraic terms say that for any equigenerated monomial ideal $I$,
one has
$$\mu(I^{k})\geq {\ell(I)+k-2\choose k-1}\mu(I)-(k-1){\ell(I)+k-2\choose k}$$ for all $k\geq 1$.
It was shown in  \cite[Theorem 2.3]{HHZ} that the equality holds if and only if $I$ is a Freiman ideal if and only if the fiber cone $F(I)$ of $I$
has minimal multiplicity if and only if $F(I)$ is Cohen-Macaulay and its defining ideal has a $2$-linear
resolution. This is a very restrictive condition for ideals arising from combinatorial structures, which
often guarantees strong combinatorial properties.

Many classes of Freiman ideals such as Hibi ideals, Veronese type ideals, matroid ideals, sortable ideals, principal Borel
ideals, $t$-spread principal Borel ideals, edge ideals and cut ideals  of several graphs,  and cover ideals of some classes of graphs such as trees, circulant graphs,  whiskered graphs and simple connected unmixed bipartite graphs have been studied (see \cite{DG,HRT,HZ1,HZ2,ZZC1,ZZC2}).

Let $S= K[x_1,\ldots,x_n]$ be the polynomial ring in $n$ variables over a field $K$ and $I\subset S$ a graded ideal.
By famous theorems   of Galligo~\cite{G} and Bayer-Stillman~\cite{BS}, the generic initial ideal of $I$ is  fixed under the action of the Borel subgroups of $GL(n,K)$.
Moreover, if  $\chara(K)=0$, then this generic initial ideal is  precisely a strongly stable ideal (also known as Borel ideal),  see \cite[Proposition 4.2.4]{HH1}.

Let $\Mon(S)$ denote the set of monomials in $S$ and for any $u\in \Mon(S)$, we set $\supp(u)=\{i:\ x_i|u\}$.
Let $k$ be a positive integer. A  monomial $x_1^{a_1}\cdots x_n^{a_n}\in \Mon(S)$ is called to be {\em $k$-bounded},  if $a_i\leq k$ for $i=1,\ldots,n$. A monomial ideal  which is generated by $k$-bounded monomials   is said to be  {\em $k$-Borel}, if  for any $k$-bounded monomial $u\in I$, and for any $j\in \supp(u)$ and $i<j$, we have $x_i(u/x_j)\in I$,  provided   $x_i(u/x_j)$ is again  $k$-bounded. Thus $k$-Borel ideals are strongly stable ideals respecting the $k$-boundedness and $1$-Borel ideals are  squarefree strongly stable  ideals, and also $1$-spread  Borel ideals.   Other interesting  restrictions of strongly stable ideals have been considered in \cite{CKST}.
For $k$-bounded monomials $u_1,\ldots,u_m$,  Herzog et al. showed  in \cite{HMRZ}  that there exists a unique smallest $k$-Borel ideal, denote by $B_k(u_1,\ldots,u_m)$, which containing $u_1,\ldots,u_m$. These  monomials $u_1,\ldots,u_m$ are called to be the {\em $k$-Borel generators} of this ideal. A $k$-Borel ideal with one $k$-Borel generator is called {\em principal $k$-Borel}.

 For a monomial ideal $I\subset S$, the unique minimal set of monomial generators
of $I$ is denoted by $G(I)$. A monomial ideal $I\subset S$ is called {\em Borel}, if for any $u\in G(I)$ and any $j\in \supp(u)$,  we
have  $x_i(u/x_j)\in  I$ for any $i<j$.  It is clear that  Borel ideals are  $k$-Borel  for $k\gg 0$, since large  $k$ imposes no conditions on the exponents of the generators. Hence  for any monomials $u_1,\ldots,u_m$, there is also   unique smallest Borel  ideal containing $u_1,\ldots,u_m$. Similarly, we denote this Borel ideal  by $B(u_1,\ldots,u_m)$ and  call these monomials to be the {\em Borel generators} of $B(u_1,\ldots,u_m)$. A Borel ideal with one Borel generator is called {\em principal Borel}.

In this paper,  we are interested in when some Borel type ideals, including Borel ideals with multiple Borel generators  and principal $k$-Borel ideals, are Freiman.

 We give a complete classification of Freiman Borel ideal of degree $2$. For the Borel ideal of high degree,  we give some sufficient conditions for it to be Freiman. The reason is that checking the Freiman condition for such ideals leads to difficult numerical problems, which at this moment, we are not able to handle.  Maybe there exists another approach to these problems which we are aware of at present.
Let  $B_k(u)$ be a principal $k$-Borel ideal of degree $d$. If $k\geq d$, then $B_k(u)=B(u)$. In this case, Herzog and the first author in  \cite[Theorem 4]{HZ2} answered when it is Freiman. If $k=1$, then $B_k(u)$ is
a  $1$-spread Borel ideal. In  this case,  a complete classification  has been given in \cite{ZZC2}. We provide  a complete classification of Freiman principal $k$-Borel ideals of degree $d$ if $k=2$ or  $k=d-1$.

Our paper is organized as follows. In Section 1, we classify all Freiman Borel ideals of degree $2$.
In Section 2,  we study Borel ideals  of degree $d\geq 3$  and ask which of them are Freiman. For  such ideals, we only have very partial results.
In Section 3, we first show that  principal $k$-Borel ideals are sortable ideals and then, appling \cite[Theorem 3]{HZ2}, we will
give a complete classification of Freiman principal $k$-Borel ideals of degree $d$ if   $k=2$ or  $k=d-1$.

Throughout this paper, we assume that $n\geq 3$ is an integer and $S= K[x_1,\ldots,x_n]$ is the polynomial ring in $n$ variables over a field $K$.

\medskip

\section{ Borel ideals of degree $2$}
 Herzog and the first author in \cite{HZ2} gave  a complete classification of Freiman principal Borel ideal of any degree.
In this section, we classify  all Freiman Borel ideals with  multiple Borel generators of degree $2$.

Let $u \in S$ be a monomial, we set $m(u)=\max\{j \;\mid \; j\in \supp(u)\}$. We need the following lemmas.

\begin{Lemma}\label{analytic spread}
	{\em (\cite[Lemma 3.1]{HZ1})}
Let $I=B(u_{1},\ldots, u_{m})$ be a Borel ideal,  where  $u_{1},\ldots, u_{m} \in  S$ are   monomials of same degree. Then $\ell(I)=\max\{m(v)\mid v\in G(I)\}$.
\end{Lemma}

 Herzog and the first author in \cite[Theorem 4]{HZ2} gave  a complete characterization of Freiman principal Borel ideals. Therefore, in the succeeding sections we will assume that 
$I=B(u_1,\ldots,u_m)$ is a Borel ideal  with Borel generators $u_1,\ldots,u_m$ and $m\geq 2$.

\begin{Lemma}\label{degree2}
Let $I=B(u_1,\ldots,u_m)$ be an  equigenerated Borel ideal of degree $2$ with Borel generators $u_1,\ldots,u_m$.
Set $u_k=x_{i_{k}}x_{j_{k}}$ with $i_k\leq j_k$ for any $1\leq k\leq m$.
Then the minimum number of Borel generators of $I$ is $m$ if and only if,
after relabelling of $u_1,\ldots,u_m$ if necessary, one has
\[
i_{1}<i_{2}<\cdots <i_{m}\leq j_{m}<j_{m-1}<\cdots<j_{1}. \eqno (1)
\]
\end{Lemma}
\begin{proof}
$(\Rightarrow)$ It is trival if $I$ is a principal Borel ideal. If the minimum number of Borel generators of $I$ is two, we may
suppose that $I=B(u_1,u_2)$, where $u_k=x_{i_k}x_{j_k}$ with $i_k\leq j_k$ for  $k=1,2$. In this case, one has $i_1\neq i_2$.
Otherwise, we may assume that $j_1>j_2$ since the case $j_1<j_2$ can be
shown by similar arguments.
Thus $u_2=x_{j_2}(u_1/x_{j_1})$, which implies that $u_2\in B(u_1)$. It follows that  $B(u_2)\subseteq B(u_1)$. Hence $I=B(u_1,u_2)=\sum_{k=1}^{2}B(u_k)=B(u_1)$, a contradiction.

Let $i_1<i_2$. If $j_1\leq j_2$, then $x_{i_2}x_{j_1}=x_{j_1}(u_2/x_{j_2})\in B(u_2)$, which forces $u_1=x_{i_1}(x_{i_2}x_{j_1}/x_{i_2})\in B(u_2)$. Hence $B(u_1)\subseteq B(u_2)$, which implies  $I=B(u_1,u_2)=B(u_2)$, a contradiction.

Assume that  the minimum number of Borel generators of $I$ is $m\geq 3$ and  $u_1,\ldots,u_m$ are Borel generators of $I$, where  $u_k=x_{i_k}x_{j_k}$ with $i_k\leq j_k$. By   minimality of the number of Borel generators, one has
$i_1,\ldots,i_m$ are different from each other. Let  $i_1<i_2<\cdots <i_m$, then, from the above arguments, one has  $j_{m}<j_{m-1}$ since $i_{m-1}<i_{m}$. Step by step, one has  $i_{1}<i_{2}<\cdots <i_{m}\leq j_{m}<j_{m-1}<\cdots<j_{1}$.

$(\Leftarrow)$ Let $i_1<\cdots<i_m\leq j_m<\cdots<j_1$, then $u_\ell\notin B(u_k)$ and $u_k\notin B(u_\ell)$
for any  two different Borel  generators $u_\ell,u_k$ of $I$. It follows that  $u_\ell\notin \sum_{k=1,\\
k\neq \ell}^{m}B(u_k)$, i.e., $u_\ell\notin B(u_1,\ldots,\hat{u_\ell},\ldots,u_m)$, as desired.
\end{proof}

\begin{Lemma}\label{case1}
Let $I=B(x_1x_{j_1},x_2x_{j_2})$ be a Borel ideal, where $2\leq j_{2}<j_1\leq n$. Then  $I$ is  Freiman.
\end{Lemma}
\begin{proof} By simple calculations, one has
$
I=(x_{1}^{2})+(x_{1}x_{2})+x_{1}(x_{3},\ldots,x_{j_{1}})+(x_{2}^2)+x_{2}(x_{3},\ldots,x_{j_{2}})
$
and
\begin{eqnarray*}
I^{2}&=&(x_{1}^{4})+(x_{1}^{3}x_{2})+x_{1}^{3}(x_{3},\ldots,x_{j_{1}})+(x_{1}^{2}x_{2}^{2})+x_{1}^{2}x_{2}(x_{3},\ldots,x_{j_{1}})
+(x_{1}x_{2}^{3})\\
&+&x_{1}^{2}(x_{3},\ldots,x_{j_{1}})^{2}+x_{1}x_{2}^{2}(x_{3},\ldots,x_{j_{1}})+x_{1}x_{2}(x_{3},\ldots,x_{j_{2}})^{2}\\
&+&x_{1}x_{2}(x_{3},\ldots,x_{j_{2}})(x_{j_{2}+1},\ldots,x_{j_{1}})+(x_{2}^{4})+x_{2}^{3}(x_{3},\ldots,x_{j_{2}})
+x_{2}^{2}(x_{3},\ldots,x_{j_{2}})^{2}.
\end{eqnarray*}
Since the minimal sets of monomial  generators   in  each  sum above are  pairwise  disjoint, it follows that
$\mu(I)=j_{1}+j_{2}-1$ and
\[
\mu(I^{2})={j_{1}-1\choose 2}+2{j_{2}-1\choose 2}+(j_{2}-2)(j_{1}-j_{2})+3j_{1}+j_{2}-3={j_{1}\choose 2}+j_{1}j_{2}.
\]
Therefore, we obtain that $I$ is Freiman by Lemma \ref{analytic spread}.
\end{proof}

\begin{Lemma}\label{case2}
Let $I=B(x_1x_{j_1},x_3^2)$ be a Borel ideal with $3< j_1\leq n$. Then  $I$ is  Freiman.
\end{Lemma}
\begin{proof}
By  direct  calculations, one has
$
I=(x_{1}^{2})+x_{1}(x_{2},x_{3})+(x_{2},x_{3})^{2}+x_{1}(x_{4},\ldots,x_{j_{1}})
$
and
\begin{eqnarray*}
I^{2}&=&(x_{1}^{4})+x_{1}^{3}(x_{2},x_{3})+x_{1}^{2}(x_{2},x_{3})^{2}+x_{1}^{3}(x_{4},\ldots,x_{j_{1}})
+x_{1}^{2}(x_{2},x_{3})(x_{4},\ldots,x_{j_{1}})\\
&+&x_{1}(x_{2},x_{3})^{3}+x_{1}^{2}(x_{4},\ldots,x_{j_{1}})^{2}+x_{1}(x_{2},x_{3})^{2}(x_{4},\ldots,x_{j_{1}})+(x_{2},x_{3})^{4}.
\end{eqnarray*}
Since the minimal sets of monomial  generators  in  each  sum above are pairwise disjoint, it follows that
$\mu(I)=j_{1}+3$ and  $\mu(I^{2})={j_{1}-2\choose 2}+6j_{1}-3=\frac{1}{2}j_1(j_1+7)$. The desired result follows from Lemma \ref{analytic spread}.
\end{proof}

\begin{Lemma}\label{case3}
Let $I=B(u_1,\ldots,u_m)$ be a Borel ideal  of degree $2$ with Borel generators $u_1,\ldots,u_m$. Let $u_k=x_{i_{k}}x_{j_{k}}$ with $1\leq i_k\leq j_k\leq n$ for $1\leq k\leq m$.  Then
\[
I=\sum_{k=1}^{m}(x_{i_{k-1}+1},\ldots,x_{i_{k}})(x_{i_{k-1}+1},\ldots,x_{j_{k}})
\]
where we stipulate $i_0=0$.
\end{Lemma}
\begin{proof}By Lemma \ref{degree2}, we may assume
$i_{1}<i_{2}<\cdots <i_{m}\leq j_{m}<j_{m-1}<\cdots<j_{1}$.
It follows that
\begin{eqnarray*}
B(u_m)&=&(x_1,\ldots,x_{i_m})(x_1,\ldots,x_{j_m})\\
&=&[(x_{1},\ldots,x_{i_{m-1}})+(x_{i_{m-1}+1},\ldots,x_{i_{m}})][(x_{1},\ldots,x_{i_{m-1}})+(x_{i_{m-1}+1},\ldots,x_{j_{m}})]\\
&=&(x_{1},\ldots,x_{i_{m-1}})^2+(x_{1},\ldots,x_{i_{m-1}})(x_{i_{m-1}+1},\ldots,x_{j_{m}})\\
&+&(x_{i_{m-1}+1},\ldots,x_{i_{m}})(x_{i_{m-1}+1},\ldots,x_{j_{m}})\\
&=&(x_{1},\ldots,x_{i_{m-1}})(x_{1},\ldots,x_{j_{m}})+(x_{i_{m-1}+1},\ldots,x_{i_{m}})(x_{i_{m-1}+1},\ldots,x_{j_{m}}).
\end{eqnarray*}
where the third equality  holds  because of $i_m\leq j_m$. Thus
\begin{eqnarray*}
I&=&\sum_{i=1}^{m}B(u_i)=\sum_{i=1}^{m-1}B(u_i)+B(u_m)\\
&=&\sum_{k=1}^{m-1}(x_{1},\ldots,x_{i_{k}})(x_{1},\ldots,x_{j_{k}})+(x_{1},\ldots,x_{i_{m-1}})(x_{1},\ldots,x_{j_{m}})\\
&+&(x_{i_{m-1}+1},\ldots,x_{i_{m}})(x_{i_{m-1}+1},\ldots,x_{j_{m}})\\
&=&\sum_{k=1}^{m-1}(x_{1},\ldots,x_{i_{k}})(x_{1},\ldots,x_{j_{k}})+(x_{i_{m-1}+1},\ldots,x_{i_{m}})(x_{i_{m-1}+1},\ldots,x_{j_{m}}).
\end{eqnarray*}
where the last equality  holds  because of  $j_m<j_{m-1}$.  By repeating the above  arguments, we obtain the desired result.
\end{proof}

\medskip
Now we are ready to present the main result of this section.
\begin{Theorem}\label{main1}
Let $m\geq 2$ be an integer and $I=B(u_1,\ldots,u_m)$  a Borel ideal  of degree $2$ with Borel generators $u_1,\ldots,u_m$ as in Lemma \ref{case3}. Then $I$ is Freiman if and only if $I=B(x_1x_{j_1},x_2x_{j_2})$ with $2\leq j_{2}<j_1\leq n$, or $I=B(x_1x_{j_1},x_3^2)$ with $3< j_1\leq n$.
\end{Theorem}
\begin{proof} $(\Leftarrow)$ It follows directly from  Lemmas \ref{case1} and \ref{case2}.
		
$(\Rightarrow)$	Let $J\!:=B(u_1,\ldots,u_{m-1})$, then, by Lemma \ref{case3}, we have
\[
J=\sum_{k=1}^{m-1}(x_{i_{k-1}+1},\ldots,x_{i_{k}})(x_{i_{k-1}+1},\ldots,x_{j_{k}})
\]
 and
\begin{eqnarray*}
 I&=&J+(x_{i_{m-1}+1},\ldots,x_{i_{m}})(x_{i_{m-1}+1},\ldots,x_{j_{m}})\\
&=&J+(x_{i_{m-1}+1},\ldots,x_{i_{m}})^2+(x_{i_{m-1}+1},\ldots,x_{i_{m}})(x_{i_{m}+1},\ldots,x_{j_{m}}).
\end{eqnarray*}
For convenience, we put $i_{1}=y$, $i_{m-1}=y+a$, $i_{m}=y+a+b$, $j_{m}=y+a+b+c$, $j_{m-1}=y+a+b+c+d$, $j_{1}=y+a+b+c+d+e$, where $y,b,d>0$ and $a,c,e\geq 0$ are integers. Moreover, $a=0$ and $e=0$ hold simultaneously.
In this case, one has  $m=2$ from relation (1).

Let  $A_{k}\!:=(x_{i_{k-1}+1},\ldots,x_{i_{k}})(x_{i_{k-1}+1},\ldots,x_{j_{k}})$ for  $1\leq k\leq m$, where $i_0=0$. Then $J=\sum_{k=1}^{m-1}A_{k}$ and $I=J+A_{m}$.
Since $\{x_{i_{m-1}+1},\ldots,x_{i_{m}}\}\cap\supp(J)=\emptyset$, we obtain that
\begin{itemize}
\item[(1)] $A_{m}^2\subseteq I^2\setminus J^2$,
\item[(2)] $(x_{1},\ldots,x_{i_{m-1}})(x_{i_{m-1}+1},\ldots,x_{i_{m}})^3\subseteq I^2\setminus J^2$,
\item[(3)] $(x_{1},\ldots,x_{i_{1}})(x_{i_{m}+1},\ldots,x_{j_{1}})(x_{i_{m-1}+1},\ldots,x_{i_{m}})^2\subseteq I^2\setminus J^2$,
\item[(4)] $(x_{1},\ldots,x_{i_{1}})(x_{i_{m-1}+1},\ldots,x_{i_{m}})(x_{i_{m}+1},\ldots,x_{j_{m}})^2\subseteq I^2\setminus J^2$,
\item[(5)] $(x_{1},\ldots,x_{i_{1}})(x_{i_{m-1}+1},\ldots,x_{i_{m}})(x_{i_{m}+1},\ldots,x_{j_{m}})(x_{j_{m}+1},\ldots,x_{j_{1}})\subseteq I^2\setminus J^2 $,
\item[(6)] $(x_{i_{1}+1},\ldots,x_{i_{m-1}})(x_{i_{m}+1},\ldots,x_{j_{m}})(x_{i_{m-1}+1},\ldots,x_{i_{m}})^2\subseteq I^2\setminus J^2 $,
\item[(7)] $(x_{i_{1}+1},\ldots,x_{i_{m-1}})(x_{j_{m}+1},\ldots,x_{j_{m-1}})(x_{i_{m-1}+1},\ldots,x_{i_{m}})^2\subseteq I^2\setminus J^2$.
\end{itemize}
Hence, one has
\begin{itemize}
\item[(a)] $ \mu(I)-\mu(J)={b+1\choose 2}+bc$,
\item[(b)] $\mu(A_{m}^2)={b+3\choose 4}+{b+2\choose 3}c+{b+1\choose 2}{c+1\choose 2}$,
\item[(c)] $\mu((x_{1},\ldots,x_{i_{m-1}})(x_{i_{m-1}+1},\ldots,x_{i_{m}})^3)=(y+a){b+2\choose 3}$,
\item[(d)] $\mu((x_{1},\ldots,x_{i_{1}})(x_{i_{m}+1},\ldots,x_{j_{1}})(x_{i_{m-1}+1},\ldots,x_{i_{m}})^2)=y(c+d+e){b+1\choose 2}$,
\item[(e)] $\mu((x_{1},\ldots,x_{i_{1}})(x_{i_{m-1}+1},\ldots,x_{i_{m}})(x_{i_{m}+1},\ldots,x_{j_{m}})^2)=yb{c+1\choose 2}$,
\item[(f)] $\mu((x_{1},\ldots,x_{i_{1}})(x_{i_{m-1}+1},\ldots,x_{i_{m}})(x_{i_{m}+1},\ldots,x_{j_{m}})(x_{j_{m}+1},\ldots,x_{j_{1}}))\!=\!ybc(d+e)$,
\item[(g)] $\mu((x_{i_{1}+1},\ldots,x_{i_{m-1}})(x_{i_{m}+1},\ldots,x_{j_{m}})(x_{i_{m-1}+1},\ldots,x_{i_{m}})^2)=ac{b+1\choose 2}$,
\item[(h)] $\mu((x_{i_{1}+1},\ldots,x_{i_{m-1}})(x_{j_{m}+1},\ldots,x_{j_{m-1}})(x_{i_{m-1}+1},\ldots,x_{i_{m}})^2)=ad{b+1\choose 2}$.
\end{itemize}
Therefore, it follows from $(b)\sim (h)$ that
\begin{eqnarray*}
\mu(I^2)-\mu(J^2)&\geq&{b+3\choose 4}+{b+1\choose 2}{c+1\choose 2}+(y+a+c){b+2\choose 3}+yb{c+1\choose 2}\\
&+&y(c+d+e){b+1\choose 2}+ybc(d+e)+a(c+d){b+1\choose 2}.\hspace{1.0cm}(2)
\end{eqnarray*}
Let $\Delta(I)\!:=\mu(I^2)-\ell(I)\mu(I)+{\ell(I)\choose 2}$, $\Delta(J)\!:=\mu(J^2)-\ell(J)\mu(J)+{\ell(J)\choose 2}$. By Lemma \ref{analytic spread}, we have $\ell(I)=\ell(J)=y+a+b+c+d+e$. It follows from (a) and (2)
that
\begin{eqnarray*}
\Delta(I)-\Delta(J)&=&[\mu(I^2)-\mu(J^2)]-\ell(I)[\mu(I)-\mu(J)]\\
&\geq & y[{b+2\choose 3}+{b+1\choose 2}(c+d+e-1)+b{c+1\choose 2}+bc(d+e-1)]\\
&+&(a+c){b+2\choose 3}+ac{b+1\choose 2}+{b+3\choose 4}+{b+1\choose 2}{c+1\choose 2}\\
&+&[ad-(a+b+c+d+e)]{b+1\choose 2}-(a+b+c+d+e)bc.
\end{eqnarray*}
Let the function on the right in the inequality above  be $f(y,a,b,c,d,e)$ and when $y\geq 2$, we
set $\Delta [f(y,a,b,c,d,e)]\!:=f(y,a,b,c,d,e)-f(y-1,a,b,c,d,e)$.
Then
\[
\Delta [f(y,a,b,c,d,e)]={b+2\choose 3}+{b+1\choose 2}(c+d+e-1)+b{c+1\choose 2}+bc(d+e-1).
\]
It follows that $\Delta f(y,a,b,c,d,e)>0$ since $b,d>0$  and $a,c,e\ge 0$, which shows that $f(y,a,b,c,d,e)$ is a strictly monotone increasing function of variable $y$.
Hence
\begin{eqnarray*}
f(y,a,b,c,d,e)&\geq& f(1,a,b,c,d,e)\\
&=&{b+2\choose 3}+{b+1\choose 2}(c+d+e-1)+b{c+1\choose 2}+bc(d+e-1)\\
&+&(a+c){b+2\choose 3}+ac{b+1\choose 2}+{b+3\choose 4}+{b+1\choose 2}{c+1\choose 2}\\
&+&[ad-(a+b+c+d+e)]{b+1\choose 2}-(a+b+c+d+e)bc\\
&=&{b+2\choose 3}(a+c+1)+{b+1\choose 2}(ad+ac-a-b-1)+{b+3\choose 4}\\
&+&b{c+1\choose 2}+{b+1\choose 2}{c+1\choose 2}-(a+b+c+1)bc.
\end{eqnarray*}
Let $g(a,b,c,d,e)\!:=f(1,a,b,c,d,e)$ and the first difference function $\Delta$ is defined by $\Delta[ g(a,b,c,d,e)]\!:=g(a,b,c,d,e)-g(a,b-1,c,d,e)$ when $b\geq 2$.
 For $i>1$, $\Delta^i$ is defined by $\Delta^{i}[g(a,b,c,d,e)]\!:=\Delta^{i-1}[\Delta[g(a,b,c,d,e)]]$ when $b\geq i+1$. We obtain
\begin{eqnarray*}
\Delta [g(a,b,c,d,e)]&=&{b+1\choose 2}(a+1)+{c+1\choose 2}(b+1)+{b+2\choose 3}+{b+1\choose 2}c\\
&-&{b\choose 2}+abd-ab-b^{2}-ac-2bc-c^{2}-b+abc,\\
\Delta^{2} [g(a,b,c,d,e)]&=&{b+1\choose 2}+{c+1\choose 2}+a(b+c+d-1)+b(c-2)-2c+1,\\
\Delta^{3} [g(a,b,c,d,e)]&=&a+b+c-2.
\end{eqnarray*}
Hence, if $b\geq 4$, then $\Delta^{3}[g(a,b,c,d,e)]>0$, which forces that $\Delta^{2}[g(a,b,c,d,e)]$ is a strictly monotone increasing function of variable $b$.
Thus, when $b\geq 3$, one has
\[
\Delta^{2}[ g(a,b,c,d,e)]\geq \Delta^{2} g(a,3,c,d,e)=2a+ac+c+{c+1\choose 2}+ad+1>0.
\]
Hence $\Delta[g(a,b,c,d,e)]$ is strictly monotonic increasing  function of variable $b$. Thus
\[
\Delta [g(a,b,c,d,e)]\geq \Delta[g(a,2,c,d,e)]={c+1\choose 2}+a+ac+2ad.\hspace{3.0cm}(3)
\]
We consider the following two cases:

(i) If $a^2+c^2=0$, then $f(1,0,b,0,d,e)=\frac{1}{24}b(b+1)(b-1)(b-2)\geq 0$. In this case, $f(1,0,b,0,d,e)=0$ if and only if $b=1$ or $b=2$.

(ii) If $a^2+c^2\neq 0$, then, by (3), we have  $\Delta [g(a,b,c,d,e)]\geq c+a+ac+2ad>0$, i.e., $g(a,b,c,d,e)\geq g(a,1,c,d,e)=ad\geq 0$. In this case, since $d>0$, one has $f(1,a,b,c,d,e)=g(a,b,c,d,e)=0$ if and only if $a=0,b=1$.

\vspace{2mm}
In  short, $\Delta(I)-\Delta(J)\geq f(1,a,b,c,d,e)\geq 0$. \hspace{5.0cm}(4)

\vspace{2mm}
Since $I$ and $J$ are equigenerated monomial ideals, one has $\Delta(I),\Delta(J)\geq 0$. Thus $I$ is Freiman if and only if   $\Delta(I)=0$. It follows that
$\Delta(I)-\Delta(J)=0$ by  (4) if and only if  $f(y,a,b,c,d,e)=0$ if and only if  $y=b=1,a=0$, or $y=1,a=c=0,b=2$ if and only if   $I=B(x_{1}x_{j_{1}},x_{2}x_{j_{2}})$, $2\leq j_{2}< j_{1}\leq n$, or $I=B(x_{1}x_{j_{1}},x_{3}^2)$, $3\leq j_{1}\leq n$. The desired results from Lemmas \ref{case1} and \ref{case2}.
\end{proof}

\medskip

\section{Equigenerated Borel ideals of degree $d\geq 3$}

In this section, let $d\geq 3$ be an integer, we give some sufficient conditions for an equigenerated Borel ideal of degree $d\geq 3$ to be Freiman ideal.

Let $u_1,\ldots,u_m\in S$ be monomials of same degree. Then $B(u_1,\ldots,u_m)$ is Freiman if and only if $x_1^kB(u_1,\ldots,u_m)$ is Freiman for any $k$. We will use this fact several times in the proof.
From Theorem \ref{main1}, one has
\begin{Remark}\label{Rek1}
Let $I=B(u_1,\ldots,u_m)$ be a Borel ideal  of degree $d$ with Borel generators $u_1,\ldots,u_m$, where $u_\ell=x_1^{d-2}v_\ell$ and $v_\ell=x_{i_{\ell}}x_{j_{\ell}}$ with $1\leq i_\ell\leq j_\ell\leq n$ for $1\leq \ell\leq m$.  Then $I$ is Freiman if and only if $I=B(x_1^{d-1}x_{j_1},x_1^{d-2}x_2x_{j_2})$ with $2\leq j_{2}<j_1\leq n$, or $I=B(x_1^{d-1}x_{j_1},x_1^{d-2}x_3^2)$ with $3< j_1\leq n$.
\end{Remark}

Let $B(u_1,\ldots,u_m)$ be  a Borel ideal  of degree $d\geq 3$ with Borel generators $u_1,\ldots,u_m$, then its
 Borel generators  have no longer the forms of Borel generators in Lemma \ref{degree2} and  the approach of Theorem \ref{main1} is no longer valid for $d\geq 3$.
On the other hand, by checking hundreds of examples, we find that $B(u_1,\ldots,u_m)$ can be Freiman only when each principal Borel ideal $B(u_k)$ is Freiman.
Therefore, from now on,  we always assume that each $B(u_k)$ is Freiman. In this case, $u_k=x_1^{d-r}x_2^{r-1}x_{i_k}$ or $x_1^{d-2}x_3^2$, where $2\leq i_k\leq n$.

\begin{Theorem}\label{main2}
 Let  $u_k=x_1^{d-r_k}x_2^{r_k-1}x_{i_k}$ with $i_k\geq 2$ for  all $1\leq k\leq m$, and let $I=B(u_1,\ldots,u_m)$ be a Borel ideal of degree $d$ with Borel generators $u_1,\ldots,u_m$. Then $I$ is Freiman.
\end{Theorem}
\begin{proof} If $r_k<d$ for all $k$, then $I=x_1^{t}B(v_1,v_2,\ldots,v_m)$, where $t=\min\{d-r_k\!: 1\leq k\leq m\}$ and $v_k=u_k/x_1^t$ for any $k$. Thus there exists some $\ell$ such that $v_\ell=x_2^{d-t}x_{i_\ell}$ and $I$ is Freiman if and only if $B(v_1,v_2,\ldots,v_m)$ is  Freiman.
Therefore, we may assume  $r_k=d$ for some $k$.
 Let  $r_1< r_2< \cdots < r_m=d$ after relabelling of $u_1,\ldots,u_m$, then $i_1> i_2>\cdots> i_m\geq 2$. It follows that
\begin{eqnarray*}
	I&=&\sum\limits_{k=1}^mB(u_k)=B(u_m)+\sum\limits_{k=1}^{m-1}x_1^{d-r_k}(x_1,x_2)^{r_k-1}(x_1,\ldots,x_{i_k})\\
&=&B(u_m)+\sum\limits_{k=1}^{m-1}x_1^{d-r_k}(x_1,x_2)^{r_k-1}\left[(x_1,\ldots,x_{i_{k+1}})+(x_{i_{k+1}+1},\ldots,x_{i_k})\right]\\
&=&\sum\limits_{k=1}^{m-1}x_1^{d-r_k}(x_1,x_2)^{r_k-1}(x_{i_{k+1}+1},\ldots,x_{i_k})+(x_1,x_2)^{d-1}(x_1,\ldots,x_{i_m})\\
&=&J+[x_1(x_1,x_2)^{d-2}+(x_2^{d-1})](x_1,\ldots,x_{i_m})\\
&=&J+x_1(x_1,x_2)^{d-2}(x_1,\ldots,x_{i_m})+(x_2^{d-1})(x_2,\ldots,x_{i_m})
\end{eqnarray*}
where the fourth equality  holds  because of $x_1^{d-r_k}(x_1,x_2)^{r_k-1}\subseteq x_1^{d-r_{k+1}}(x_1,x_2)^{r_{k+1}-1}$ for any $1\leq k\leq m-1$ and $J=\sum\limits_{k=1}^{m-1}x_1^{d-r_k}(x_1,x_2)^{r_k-1}(x_{i_{k+1}+1},\ldots,x_{i_k})$.

Let $I_1\!:=J+x_1(x_1,x_2)^{d-2}(x_1,\ldots,x_{i_m})$, $I_2\!:=x_2^{d-1}(x_2,\ldots,x_{i_m})$. Since $x_1\in \supp(G(I_1))\setminus\supp(G(I_2))$, one has  $G(I_1)\cap G(I_2)=\emptyset$, $G(I_1^2)\cap G(I_2^2)=\emptyset$ and  $G(I_{1}I_{2})\cap G(I_2^2)=\emptyset$.
Thus
$$\mu(I_2)=\mu(I)-\mu(I_1)=i_m-1 \eqno(5)$$
and
\begin{eqnarray*}
	I^2&=&(I_1+I_2)^2=I_1^2+I_1I_2+I_2^2=I_{1}^2+[J+x_1(x_1,x_2)^{d-2}(x_1,\ldots,x_{i_m})]I_2+I_2^2\\
	&=&I_1^2+JI_2+x_1(x_1,x_2)^{d-2}(x_1,\ldots,x_{i_m})I_2+I_2^2\\
	&=&I_1^2+JI_2+x_1^2(x_1,x_2)^{d-2}I_2+x_1(x_1,x_2)^{d-2}(x_2,\ldots,x_{i_m})I_2+I_2^2\\
	&=&I_1^2+[J+x_1x_2^{d-2}(x_2,\ldots,x_{i_m})]I_2+I_2^2
\end{eqnarray*}
where the last equality  holds  because of $x_1^2(x_1,x_2)^{d-2}\subseteq J$.

Since $G(I_1^2)$, $G(I_2^2)$ and  $G((J+x_1x_2^{d-2}(x_2,\ldots,x_{i_m}))I_2)$ are pairwise disjoint,  for any  different $k,\ell$, $G(x_1x_2^{d-2}(x_2,\ldots,x_{i_m}))$, $G(x_1^{d-r_k}(x_1,x_2)^{r_k-1}(x_{i_{k+1}+1},\ldots,x_{i_k}))$ and $G(x_1^{d-r_\ell}(x_1,x_2)^{r_\ell-1}(x_{i_{\ell+1}+1},\ldots,x_{i_\ell}))$   are also pairwise disjoint. It follows
 that
\begin{eqnarray*}\mu(I^2)-\mu(I_1^2)&=&\mu((J+x_1x_2^{d-2}(x_2,\ldots,x_{i_m}))I_2)+\mu(I_2^2)\\
	&=&\mu(JI_2)+\mu(x_1x_2^{d-2}(x_2,\ldots,x_{i_m})I_2)+\mu(I_2^2)\\
&=&\sum_{k=1}^{m-1}(i_k-i_{k+1})(i_m-1)+2{i_m\choose 2}\\
&=&i_1(i_m-1).
\end{eqnarray*}
Let $\Delta(I)\!:=\mu(I^2)-\ell(I)\mu(I)+{\ell(I)\choose 2}$ and  $\Delta(I_1)\!:=\mu(J^2)-\ell(J)\mu(J)+{\ell(J)\choose 2}$. Then,
by Lemma \ref{analytic spread} and (5), one has
\[
\Delta(I)-\Delta(I_1)=\mu(I^2)-\mu(I_1^2)-\ell(I)[\mu(I)-\mu(I_1)]=0.
\] Therefore, $I$ is Freiman if and only if $I_1$ is Freiman.

We will show that $I_1$ is Freiman by induction on $d$.
 If $d=3$, then $I=B(x_1^2x_{i_1},\\ x_1x_2x_{i_2})
 =x_1B(x_1x_{i_1},x_2x_{i_2})$ with $i_1>i_2$, or $I=B(x_1^2x_{i_1},x_1x_2x_{i_2},x_2^2x_{i_3})=B(x_1^2x_{i_1},\\ x_1x_2x_{i_2})+x_2^2(x_{2},\cdots,x_{i_3})$ with $i_1>i_2>i_3$.
In the first case,  $I$ is Freiman by Remark \ref{Rek1}. It forces that in another case $I$ is also Freiman by above arguments.
Now, assume that $I=B(u_1,\ldots,u_m)$ where $u_m=x_2^{d-1}x_{i_m}$, then $r_1< r_2< \cdots <r_{m-1}<d$. Thus one has
\begin{eqnarray*}
I&=&B(u_1,\ldots,u_m)=B(u_1,\ldots,u_{m-1})+B(u_m)\\
	&=&B(u_1,\ldots,u_{m-1})+B(x_1x_2^{d-2}x_{i_m})+x_2^{d-1}(x_{2},\cdots,x_{i_m})\\
&=&x_1B(v_1,\ldots,v_{m-1},x_2^{d-2}x_{i_m})+x_2^{d-1}(x_{2},\cdots,x_{i_m}).
\end{eqnarray*}
 where the third equality  holds  because of $B(u_m)=B(x_1x_2^{d-2}x_{i_m})+x_2^{d-1}(x_{2},\cdots,x_{i_m})$ and  $v_k=u_k/x_1$ for any $1\leq k\leq m-1$.
 By induction hypothesis, one has  $x_1B(v_1,v_2,\ldots,v_{m-1},x_2^{d-2}x_{i_m})$ is Freiman. It forces that $I$ is Freiman
  by  above arguments.
\end{proof}

The following example shows that when $x_1^{d-2}x_3^2$ is  one of the Borel generators of $B(u_1,\ldots,u_m)$, we can not determine whether $B(u_1,\ldots,u_m)$ is Freiman.
\begin{Example}
	For any integer $d\geq 2$, $B(x_1^{d-1}x_4,x_1^{d-2}x_3^2)$ is Freiman by Remark \ref{Rek1}.
	However, Borel ideal $B(x_1x_3^2,x_2^2x_4)=(x_1^3,x_1^2x_2,x_1^2x_3,x_1^2x_4,x_1x_2^2,x_1x_2x_3,x_1x_2x_4,x_1x_3^2,\\
   x_2^3,x_2^2x_3,x_2^2x_4)$. Thus $\mu(I)=11$ and $\ell(I)=4$ by Lemma \ref{analytic spread}.  By using CoCoA, we obtain  $\mu(I^2)=41$. Hence
	$\mu(I^2)>\ell(I)\mu(I)-{\ell(I) \choose 2}=38$, i.e., $B(x_1x_3^2,x_2^2x_4)$
	is not Freiman.
\end{Example}

\medskip

\section{Principal $k$-Borel ideals}

In this section,  we first show that  principal $k$-Borel ideals are sortable ideals and then, appling \cite[Theorem 3]{HZ2}, we will
give a complete classification of Freiman principal $k$-Borel ideals of degree $d$ if   $k=2$ or  $k=d-1$.

Let $d$ be a positive integer, $S_d$ the $K$-vector space generated by the monomials of degree $d$ in $S$, and take two monomials $u,v\in S_d$. We
write $uv=x_{i_1}x_{i_2}\cdots x_{i_{2d}}$ with $1\leq i_1\leq i_2\leq \cdots \leq i_{2d}\leq n$,  and define
$$u'=x_{i_1}x_{i_3}\cdots x_{i_{2d-1}},\quad \text{and} \quad v'=x_{i_2}x_{i_4}\cdots x_{i_{2d}}.$$
The pair $(u',v')$ is called the {\it sorting} of $(u,v)$.
In this way we obtain a map
\[
\sort: S_d\times S_d \to S_d\times S_d,\ (u, v)\mapsto (u', v').
\]
 A pair $(u,v)$ is called to be  {\em sorted} if $\sort(u,v)=(u,v)$ or
$\sort(u,v)=(v,u)$, otherwise it is called to {\em unsorted}. Notice that $\sort(u,v)=\sort(v,u)$.

\begin{Definition} \label{HHF}{\em A subset $B\subset S_d$ of monomials is called {\em sortable} if
$\sort(B\times B)\subset B\times B$}. An equigenerated monomial ideal $I\subset S$ is said to be   {\em sortable}, if $G(I)$ is a sortable set.
\end{Definition}

For a monomial ideal $I$, we set $I^{\leq k}=(u\in G(I): \ u\  \textrm{is $k$-bounded})$.

\begin{Theorem}\label{sort}  Let $u$ be a $k$-bounded monomial, and  $I=B_{k}(u)$  a  $k$-Borel ideal, then $I$ is  sortable.
\end{Theorem}
\begin{proof} From \cite[Lemma 1.4]{HMRZ}, one has $I=B_{k}(u)=B(u)^{\leq k}$. Hence, it is enough to show that $B(u)$ is sortable.
Let $u=x_{i_1}x_{i_2}\cdots x_{i_d}$ with $i_1\leq i_2\leq \cdots \leq i_d$, and $v_1=x_{j_1}x_{j_2}\cdots x_{j_d}$,  $v_2=x_{\ell_1}x_{\ell_2}\cdots x_{\ell_d}\in G(B(u))$ with $j_1\leq j_2\leq \cdots \leq j_d$
and  $\ell_1\leq \ell_2\leq \cdots \leq \ell_d$. Then $j_t,\ell_t\leq i_t$ for any $t$. Let $v_1v_2=x_{s_1}x_{s_2}\cdots x_{s_{2d}}$ with $s_1\leq s_2\leq\cdots\leq s_{2d}$.
Then we have  $s_{2t-1}\leq s_{2t}\leq i_t$ for $1\leq t\leq d$. It follows that  $v'_1=x_{s_1}x_{s_3}\cdots x_{s_{2d-1}}\in G(B(u))$ and  $v'_2=x_{s_2}x_{s_4}\cdots x_{s_{2d}}\in G(B(u))$, as wished.
\end{proof}

\medskip
Given a sortable ideal $I$, we can  associate a graph, denoted by  the {\em sorted graph} $G_{I,s}$ of $I$, its vertex set is  $G(I)$ and edge set is
$E(G_{I,s})=\{\{u,v\}\: \text{$u\neq v$, $(u,v)$  is sorted}\}$.
The following lemma  is the main tool we will use later.

\begin{Lemma}	\label{judge}
	{\em (\cite[Theorem 3]{HZ2})}
		Let $I$ be a sortable ideal. Then $I$ is  Freiman  if and only if the sorted graph $G_{I,s}$ of $I$ is chordal. In particular, if $G_{I,s}$
	contains an induced  $t$-cycle with $t\geq 4$, then $I$ is not Freiman.
\end{Lemma}

\begin{Theorem}\label{$k$-Borel}
	Let $B(u)$ be a principal Borel ideal with Borel generator  $u$. If $B(u)$ is Freiman, then  $B_{k}(u)$ is Freiman for any $k$.
\end{Theorem}
\begin{proof}
Since $G(B_{k}(u))\subseteq G(B(u))$, the sorted graph $G_{B_{k}(u),s}$ of $B_{k}(u)$ is the induced subgraph of the sorted graph $G_{B(u),s}$. If $B(u)$ is Freiman, then
$G_{B(u),s}$ is a chordal graph by Lemma \ref{judge}. Obviously, the induced subgraph $G_{B_{k}(u),s}$ is also a chordal graph. Again using Lemma \ref{judge}, we obtain that $B_{k}(u)$ is Freiman.
\end{proof}

Let  $B_k(u)$ be a principal $k$-Borel ideal of degree $d$. If $k\geq d$, then $B_k(u)=B(u)$. In this case, Herzog and the first author in  \cite[Theorem 4]{HZ2} answered when it is Freiman. If $k=1$, then $B_k(u)$ is 
a  $1$-spread Borel ideal. In  this case,  a complete classification  has been given in \cite{ZZC2}.  Next, we only focus on the cases $k=d-1$ and $k=2$.

\begin{Theorem}\label{degreed}
	Let  $u$ be a $(d-1)$-bounded monomial of degree $d$. Then  $(d-1)$-Borel ideal $B_{d-1}(u)$ is Freiman if and only if the Borel ideal  $B(u)$ is Freiman.
\end{Theorem}
\begin{proof}
$(\Leftarrow)$ It follows from Theorem 	\ref{$k$-Borel}.

$(\Rightarrow)$ Since $B(u)$ is a Borel ideal, one has  $x_1^d\in B(u)$. Thus $G(B_{d-1}(u))=G(B(u))\setminus \{x_1^{d},x_2^{d},\ldots,x_m^{d}\}$  where $m=max\{i\!: x_i^d\in G(B(u))\}$.
For any  $1\leq j\leq m$ and any $v\in G(B(u))$, we set $k=\left\{\begin{array}{ll}
0\ &\text{if}\ \ x_j\nmid v\\
\max\{k\,: x_j^k|v\}\ & \text{otherwise}
\end{array}\right.$. By the definition of the sorted, we obtain that the pair $(v,x_j^d)$ is sorted if and only if  $k\geq d-1$. Hence
 $\{v,x_j^d\}$ is an edge of the  sorted graph $G_{B(u),s}$ of $B(u)$ if and only if $k=d-1$.

If  $B(u)$ is not Freiman, then $G_{B(u),s}$ contains an induced $t$-cycle $C$ of length $t\geq 4$. we label  its vertex as $v_1,v_2,\ldots,v_{t}$ in clockwise order. For any
 $1\leq j\leq m$, we claim $x_j^d\notin \{v_1,v_2,\ldots,v_{t}\}$. Thus  $C$ is also a  induced $t$-cycle of the sorted graph $G_{B_{d-1}(u),s}$, a contradiction.

The proof of the above assertion: If $x_j^d\in \{v_1,v_2,\ldots,v_{t}\}$, we may assume $v_1=x_j^d$. Then $v_2=x_j^{d-1}x_a$ and  $v_t=x_j^{d-1}x_b$ since $\{v_1,v_2\},\{v_1,v_t\}\in E(G_{B(u),s})$. This implies $\{v_2,v_t\} \in E(G_{B(u),s})$, which contradicts that $C$ is a circle of  length $\geq 4$.
\end{proof}

From the above theorem and \cite[Theorem 4]{HZ2}, one has the following result.

\begin{Corollary}\label{C1}
	Let  $u$ be a $(d-1)$-bounded monomial of degree $d$.
	\begin{enumerate}
		\item [(1)] If $d=2$, then $1$-Borel ideal $B_{1}(u)$ is Freiman if and only if $u=x_1x_i$ with $2\leq i\leq n$, or $u=x_2x_j$ with  $3\leq j\leq n$.
		\item[(2)] If $d=3$, then $2$-Borel ideal $B_{2}(u)$ is Freiman if and only if $u=x_1x_3^{2}$, $u=x_2x_3^{2}$, or  $u$  is a minimal  monomial  generators of
		$x_1^{2}(x_2,\ldots,x_n)$ or $x_{2}^{2}(x_3,\ldots,x_n)$ or   $x_{1}x_{2}(x_3,\ldots,x_n)$.
	\item[(3)] If $d\geq 4$, then $(d-1)$-Borel ideal $B_{d-1}(u)$ is Freiman if and only if  $u=x_{1}^{d-2}x_{3}^2$,  or $u$  is a minimal  monomial  generators of
     $x_1^{d-1}(x_2,\ldots,x_n)$ or $x_{2}^{d-1}(x_3,\ldots,x_n)$ or   $x_{1}^{d-r-1}x_{2}^{r}(x_3,\ldots,x_n)$ where $1\leq r\leq d-2$.
     \end{enumerate}
\end{Corollary}

The following example shows that the assumption in Theorem \ref{degreed} that  $u$ is a $(d-1)$-bounded monomial of degree $d$
cannot be replaced by the condition that $u$ is a $k$-bounded monomial of degree $d$ with $k\leq d-2$.
\begin{Example}
	Let $u=x_2x_3x_4$ be a $1$-bounded monomial, then the principal $1$-Borel ideal $B_1(u)
	=(x_1x_2x_3,x_1x_2x_4,x_1x_3x_4,x_2x_3x_4)$. Its  sorted graph  $G_{B_1(u),s}$ is shown in Figure $1$. Hence $B_1(u)$ is Freiman  by Lemma \ref{judge}.
	However, $B(x_2x_3x_4)$ is not Freiman by \cite[Theorem 4]{HZ2}.
\end{Example}
\begin{center}
	\setlength{\unitlength}{1.0mm}
	\begin{picture}(85,40)
		\linethickness{1pt}
		\thicklines
		\multiput(-25,11)(20,0){2}{\circle*{1}}
		\multiput(-25,31)(20,0){2}{\circle*{1}}
		\multiput(-25,11)(20,0){2}{\line(0,1){20}}
		\multiput(-25,11)(0,20){2}{\line(1,0){20}}
		\put(-25,11){\line(1,1){20}}
		\put(-25,31){\line(1,-1){20}}
		\put(-30,33){$x_1x_2x_3$}
		\put(-10,33){$x_1x_2x_4$}
		\put(-30,7){$x_1x_3x_4$}
		\put(-10,7){$x_2x_3x_4$}
		\put(-23,0){$Figure\ 1$}
		
		\multiput(17,13)(20,0){2}{\circle*{1}}
		\multiput(17,23)(20,0){2}{\circle*{1}}
		\put(27,33){\circle*{1}}
		\multiput(17,13)(20,0){2}{\line(0,1){10}}
		\multiput(17,13)(0,10){2}{\line(1,0){20}}
		\put(17,23){\line(1,1){10}}
		\put(27,33){\line(1,-1){10}}
		\put(17,13){\line(2,1){20}}
		\put(25,35){$x_1^2x_3^2$}
		\put(4,23){$x_1x_2x_3^2$}
		\put(9,8){$x_1x_2^2x_3$}
		\put(35,8){$x_1^2x_2^2$}
		\put(39,23){$x_1^2x_2x_3$}
		\put(19,0){$Figure\ 2$}
		
		\multiput(80,13)(10,0){2}{\circle*{1}}
		\multiput(70,23)(30,0){2}{\circle*{1}}
		\multiput(80,33)(10,0){2}{\circle*{1}}
		\put(70,23){\line(1,1){10}}
		\put(70,23){\line(1,-1){10}}
		\put(70,23){\line(2,1){20}}
		\put(70,23){\line(1,0){30}}
		\put(80,13){\line(0,1){20}}
		\put(80,13){\line(2,1){20}}
		\put(80,13){\line(1,0){10}}
		\put(100,23){\line(-1,1){10}}
		\put(100,23){\line(-1,-1){10}}
		\put(57,23){$x_1x_2x_3^2$}
		\put(75,8){$x_1x_2^2x_3$}
		\put(90,8){$x_1^2x_2^2$}
		\put(101,23){$x_1^2x_2x_3$}
		\put(87,35){$x_1^2x_3^2$}
		\put(77,35){$x_2^2x_3^2$}
		\put(78,0){$Figure\ 3$}
	\end{picture}
\end{center}

\medskip
\begin{Lemma}\label{isomorphic}
	Let  $u=x_1^kx_{i_{1}}x_{i_{2}}\cdots x_{i_{d-k}}$ be a $k$-bounded monomial of degree $d$, where $2\leq i_1\leq i_2\leq \cdots\leq  i_{d-k}$. Define a map $\psi_k$ as follow:
	\begin{align*}
		\psi_k:S&\longrightarrow S,\\
		u&\longmapsto x_{i_1-1}x_{i_2-1}\cdots x_{i_{d-k}-1}.
	\end{align*}
	Then $B_k(u)$ is Freiman if and only if $B_k(\psi_k(u))$ is Freiman.
\end{Lemma}
\begin{proof} For any $v_1,v_2\in G(B_k(u))$, one has  $x_1^k|\gcd(v_1,v_2)$.
Hence we may assume that $v_1=x_1^kx_{j_{1}}x_{j_{2}}\cdots x_{j_{d-k}}$,  $v_2=x_1^kx_{\ell_{1}}x_{\ell_{2}}\cdots x_{\ell_{d-k}}$ with $2\leq j_1\leq j_2\leq \cdots\leq  j_{d-k}$ and $2\leq \ell_1\leq \ell_2\leq \cdots\leq \ell_{d-k}$.
It is clear that the pair $(v_1,v_2)$ is sorted if and only if the pair $(\psi_k(v_1),\psi_k(v_2))$ is sorted by the definition of $\psi_k$. Hence the sorted graphs $G_{B_k(u),s}$ and~$G_{B_k(\psi_k(u)),s}$ are isomorphic. The desired result follows from  Lemma \ref{judge}.
\end{proof}

According to the above theorem, suppose that $u$ is a $k$-bounded monomial. Next, we usually assume  $x_1^k\nmid u$.

\begin{Theorem}\label{d=4}
	Let  $u$ be a $2$-bounded monomial of degree $4$. Then $B_2(u)$ is Freiman if and only if $u=x_1x_2x_3^2$, or $u=x_2^2x_3^2$, or $u=x_1x_2^2x_{i_4}$ with  $ i_4\geq 3$.
\end{Theorem}
\begin{proof}$(\Leftarrow)$ Since   $G(B_2(x_1x_2x_3^2))\!=\!\{x_1^2x_2^2,x_1^2x_2x_3,x_1^2x_3^2,x_1x_2^2x_3,x_1x_2x_3^2\}$, 	
	 $G(B_2(x_2^2x_3^2))\\
	 =\{x_1^2x_2^2,x_1^2x_2x_3,x_1^2x_3^2,x_1x_2^2x_3,x_1x_2x_3^2,x_2^2x_3^2\}$. In these two cases, their sorted grphas are Figures $2$ and $3$, respectively.
	Obviously, they are chordal graphs, hence  $B_2(x_1x_2x_3^2)$ and $B_2(x_2^2x_3^2)$ are Freiman by  Lemma \ref{judge}.
	
	If $u=x_1x_2^2x_{i_4}$ with  $ i_4\geq 3$, then $B(u)$ is Freiman by \cite[Theorem 4 (c)]{ HZ2}. It follows that $B_2(u)$ is Freiman by Theorem \ref{$k$-Borel}.

If $u$ is a $2$-bounded monomial different from the above monomials, then it is one of the following seven  cases:
(1) $u=x_1x_{2}x_{3}x_{i_4}$ with  $i_4\geq 4$; (2) $u=x_1x_{2}x_{i_3}x_{i_4}$ with $4\leq i_3\leq i_4 $;
(3) $u=x_1x_{i_2}x_{i_3}x_{i_4}$ with $3\leq i_2\leq i_3\leq i_4 $;
(4)  $u=x_{2}^2 x_{3}x_{i_4}$ with  $i_4\geq 4$; (5) $u=x_{2}^2x_{i_3}x_{i_4}$ with $4\leq i_3\leq i_4$;
(6) $u=x_{2}x_{i_2}x_{i_3}x_{i_4}$ with $3\leq i_2\leq i_3\leq i_4$; (7) $u=x_{i_1}x_{i_2}x_{i_3}x_{i_4}$ with $3\leq i_1\leq i_2\leq i_3\leq i_4$.
No matter which case  $u$ is, we obtain that
$G_{B_2(u),s}$ contains an induced $4$-cycle with vertex set  $\{x_1^2x_2^2,x_1^2x_2x_4,x_1x_2x_3x_4,x_1x_2^2x_3\}$. Hence
$B_2(u)$ is not Freiman by  Lemma \ref{judge}.
\end{proof}

\medskip
A graph $G$ is called $3$-partite if its  vertex set $V$ can be partitioned into disjoint union $V=V_1\sqcup V_2\sqcup V_3$.

\begin{Lemma}\label{chordal}
	Let $p$ be a positive integer, $G$ be  a $3$-partite graph with vertex set $V=V_1\sqcup V_2\sqcup V_3$ and edge set $E(G)$ where $V_1=\{a_1,\ldots,a_p\}$,  $V_2=\{b_1,\ldots,b_p\}$ and  $V_3=\{c_1,\ldots,c_p\}$. If $G$
satisfies the following conditions:
\begin{itemize}
 	\item[(1)] The induced subgraph of $G$ on $V_i$ is a complete graph for any $1\leq i\leq 3$;
    \item[(2)] $\{a_i,b_j\}\in E(G)$ if and only if $i\leq j$;
    \item[(3)]$\{b_i,c_j\}\in E(G)$ if and only if $i\leq j$;
    \item[(4)] $\{a_i,c_j\}\notin E(G)$ for any $1\leq i,j\leq p$.
  \end{itemize}	
Then $G$ is chordal.  In particular, any induced subgraph of $G$ is also chordal.
\end{Lemma}
\begin{proof} It is trival for the case $p=1,2$. Now, we assume that $p\geq 3$. If $G$ contains  an induced $t$-cycle $C$ of length $t\geq 4$, then we label its vertex as $v_1,v_2,\ldots,v_t$ in clockwise order.
By the hypothesis (1), one has $|\{v_1,\ldots,v_t\}\cap V_i|\leq 2$ for any $i$, which forces that $t\leq 6$. If $t=5$ or $6$, then there are at least a vertex, say $v_1$, belongs to $V_1$. It follows that $v_2\in V_1$
or $v_2\in V_2$ by the supposition (4). If $v_2\in V_1$, then, by the supposition (1), one has  $v_3,v_t\in V_2$ and $\{v_3,v_t\}\in E(G)$. This contradicts the assumption that $C$ is a circle  of length $t$.
If $v_2\in V_2$, then  we can also arrive at contradictions by similar arguments. Next, let  $t=4$, we consider the following two cases:

(1) If $\{v_1,\ldots,v_4\}\cap V_i\neq \emptyset$ for any $i$. Let $v_1\in V_1$, then $v_3\in V_3$ by the supposition (4). It follows that $v_2,v_3\in V_2$ since $\{v_2,v_3\},\{v_3,v_4\}\in E(G)$. This implies that
$\{v_2,v_4\}\in E(G)$, a contradiction.

(2) If there exists some $i$ such that $\{v_1,\ldots,v_4\}\cap V_i=\emptyset$, then $i\neq 2$, and $|\{v_1,\ldots,v_4\}\cap V_1|=|\{v_1,\ldots,v_4\}\cap V_2|=2$ or $|\{v_1,\ldots,v_4\}\cap V_3|=|\{v_1,\ldots,v_4\}\cap V_2|=2$. We may assume that the former case holds and $v_1,v_2\in V_1$, then $v_3,v_4\in V_2$. Thus $v_1=a_i,v_2=a_j$, $v_3=b_\ell,v_4=b_m$ with $i\leq m$ and $j\leq \ell$. Since $\{v_1,v_3\}=\{a_i,b_\ell\}\in E(G)$, we have $\ell<i$. Hence $j<m$ and $\{v_2,v_4\}\in E(G)$, a contradiction.
\end{proof}

Now, we stipulate that $\prod\limits_{j=l}^{m}x_{j}=1$ if $l>m$.

 \begin{Theorem}\label{general}
 Let $d\geq 5$ be an integer and  $u$  one of the following $2$-bounded monomials of degree $d$:
 \begin{itemize}
 	\item[(1)] $u=\left\{\begin{array}{ll}
 		x_1(\prod_{i=2}^{m}x_i^2)x_{m+1}\ &\text{if}\ d=2m\\
 		x_1(\prod_{i=2}^{m+1}x_i^2)\ & \text{if}\ d=2m+1
 	\end{array}\right.$;
 		\item[(2)]
 $u=\left\{\begin{array}{ll}
 		x_1(\prod_{i=2}^{m}x_i^2)x_{m+2}\ &\text{if}\ d=2m\\
 		x_1(\prod_{i=2}^{m}x_i^2)x_{m+1}x_{m+2}\ & \text{if}\ d=2m+1
 	\end{array}\right.$;
 		\item[(3)]$u=\left\{\begin{array}{ll}
 		x_1(\prod_{i=2}^{m-1}x_i^2)x_{m}x_{m+1}^2 &\text{if}\ d=2m\\
 		x_1(\prod_{i=2}^{m}x_i^2)x_{m+2}^2 & \text{if}\ d=2m+1
 	\end{array}\right.$.
  \end{itemize}	
 	Then $B_2(u)$ is Freiman.
\end{Theorem}
\begin{proof} (1) Let $v_1,v_2\in G(B_2(u))$. If $d=2m+1$, then $v_1=(\prod_{i=1}^{p-1}x_i^2)x_{p}(\prod_{i=p+1}^{m+1}x_i^2)$,  $v_2=(\prod_{i=1}^{q-1}x_i^2)x_{q}(\prod_{i=q+1}^{m+1}x_i^2)$ for some $1\leq p,q\leq m+1$, or vice versa. If $d=2m$, then $v_1=\prod_{i=1}^{m}x_i^2$ or $(\prod_{i=1}^{s-1}x_i^2)x_{s}(\prod_{i=s+1}^{m}x_i^2)x_{m+1}$,  $v_2=(\prod_{i=1}^{t-1}x_i^2)x_{t}(\prod_{i=t+1}^{m}x_i^2)x_{m+1}$ for some $1\leq s,t\leq m$, or vice versa.
 No matter which case, one has  $\text{sort}(v_1,v_2)=(v_1,v_2)$ or $\text{sort}(v_1,v_2)=(v_2,v_1)$. This forces that
 the sorted graph $G_{B_2(u),s}$  is a complete graph,  hence $B_2(u)$ is Freiman by  Lemma \ref{judge}.

(2) If $d=2m$, then for any $v\in G(B_2(u))$, one has $v=(\prod_{i=1}^{s-1}x_i^2)x_s(\prod_{i=s+1}^{m}x_i^2)x_{m+1}$, or  $\prod_{i=1}^{m}x_i^2$, or
$(\prod_{i=1}^{t-1}x_i^2)x_t(\prod_{i=t+1}^{m}x_i^2)x_{m+2}$ with $1\leq s,t\leq m$.
Put $V_1=\{a_1,\ldots,a_{m+1}\}$,  $V_2=\{b_1,\ldots,b_{m}\}$, where $a_s=(\prod_{i=1}^{s-1}x_i^2)x_s(\prod_{i=s+1}^{m}x_i^2)x_{m+1}$,  $a_{m+1}=\prod_{i=1}^{m}x_i^2$ and
$b_t=(\prod_{i=1}^{t-1}x_i^2)x_t(\prod_{i=q+1}^{m}x_i^2)x_{m+2}$ for $1\leq s,t\leq m$. Then $V_1\sqcup V_2$ is the vertex set of the sorted graph $G_{B_2(u),s}$
 and it is easy to verify that the induced subgraph of $G_{B_2(u),s}$ on $V_i$ is a complete graph for  $i=1,2$.
For $1\leq s,t\leq m$, we obtain that  $\{a_s,b_{t}\}\in E(G_{B_2(u),s}))$  if and only if $s\geq
t$ and $\{a_{m+1},b_{t}\}\in E(G_{B_2(u),s}))$   by the definition of sorted. 
Thus $G_{B_2(u),s}$ is  isomorphic to an induced subgraph of  $G$ in Lemma \ref{chordal}  with $p=m+1$,
hence $B_2(u)$ is Freiman by Lemma \ref{judge}.

If $d=2m+1$, then for any $v\in G(B_2(u))$, one has $v=(\prod_{i=1}^{m}x_i^2)x_{m+2}$, or $(\prod_{i=1}^{p-1}x_i^2)x_p(\prod_{i=p+1}^{m}x_i^2)x_{m+1}x_{m+2}$, or
$(\prod_{i=1}^{q-1}x_i^2)x_q(\prod_{i=q+1}^{m+1}x_i^2)$ with $1\leq p\leq m$,
$1\leq q\leq m+1$. Let  $V_1=\{a_1,\ldots,a_{m+1}\}$, $V_2=\{b_1,\ldots,b_{m+1}\}$, where 
$a_p=(\prod_{i=1}^{p-1}x_i^2)x_p(\prod_{i=p+1}^{m}x_i^2)x_{m+1}x_{m+2}$, $a_{m+1}=(\prod_{i=1}^{m}x_i^2)x_{m+2}$ and $b_q=(\prod_{i=1}^{q-1}x_i^2)x_q\cdot (\prod_{i=q+1}^{m+1}x_i^2)$
 for $1\leq p\leq m$, $1\leq q\leq m+1$. Then $V_1\sqcup V_2$ is the vertex set of the sorted graph $G_{B_2(u),s}$.
Similar to arguments as the case $d=2m$, we get that the induced subgraph of $G_{B_2(u),s}$ on $V_i$ is a complete graph for  $i=1,2$ and $\{a_p,b_{q}\}\in E(G_{B_2(u),s}))$  if and only if $p\leq q$.
Therefore, $B_2(u)$ is Freiman by Lemmas \ref{chordal} and \ref{judge}.

(3) If $d=2m$, then for any $v\in G(B_2(u))$, we have $v=\prod_{i=1}^{m}x_i^2$, or
 $(\prod_{i=1}^{m-1}x_i^2)x_{m+1}^2$ or
$(\prod_{i=1}^{p-1}x_i^2)x_p(\prod_{i=p+1}^{m-1}x_i^2)
x_{m}x_{m+1}^2$, or $(\prod_{i=1}^{q-1}x_i^2)x_q(\prod_{i=q+1}^{m}x_i^2)x_{m+1}$, where
$1\leq p\leq m-1$, and $1\leq q\leq m$.
Put $V_1=\{a_1,\ldots,a_{m}\}$, $V_2=\{b_0,b_1,\ldots,b_{m}\}$, where $a_p=(\prod_{i=1}^{p-1}x_i^2)x_p(\prod_{i=p+1}^{m-1}x_i^2)x_{m}x_{m+1}^2$,
$a_m=(\prod_{i=1}^{m-1}x_i^2)x_{m+1}^2$, $b_0=\prod_{i=1}^{m}x_i^2$ and
$b_q=(\prod_{i=1}^{q-1}x_i^2)x_q\cdot (\prod_{i=q+1}^{m}x_i^2)x_{m+1}$ for $1\leq p\leq m-1$
and $1\leq q\leq m$. Then $V_1\sqcup V_2$ is the vertex set of the sorted graph
 $G_{B_2(u),s}$ and the induced subgraph of
 $G_{B_2(u),s}$  on $V_i$ is a complete graph for $i=1,2$ and $\{a_p,b_q\}\in E(G_{B_2(u),s})$
 if and only if $p\leq q$. Therefore, $B_2(u)$ is Freiman by Lemmas \ref{chordal} and \ref{judge}

If $d=2m+1$, then for any $v\in G(B_2(u))$, we have $v=(\prod_{i=1}^{m}x_i^2)x_{m+1}$, or
$(\prod_{i=1}^{m}x_i^2)x_{m+2}$, or
$(\prod_{i=1}^{p-1}x_i^2)x_p(\prod_{i=p+1}^{m}x_i^2)x_{m+2}^2$, or
$(\prod_{i=1}^{q-1}x_i^2)x_q(\prod_{i=q+1}^{m}x_i^2)x_{m+1}x_{m+2}$, or
$(\prod_{i=1}^{r-1}x_i^2)\cdot x_r(\prod_{i=r+1}^{m}x_i^2)x_{m+1}^2$,
where $1\leq p,q,r\leq m$.
Set $V_1=\{a_1,\ldots,a_{m-1}\}$, $V_2=\{b_1,\ldots,b_{m+1}\}$ and $V_3=\{c_1,\ldots,
c_{m+1}\}$, where $a_p=(\prod_{i=1}^{p-1}x_i^2)x_p(\prod_{i=p+1}^{m}x_i^2)x_{m+2}^2$,
$b_q=(\prod_{i=1}^{q-1}x_i^2)x_q(\prod_{i=q+1}^{m}x_i^2)x_{m+1}x_{m+2}$, $b_{m+1}=(\prod_{i=1}^{m}x_i^2)x_{m+2}$,
$c_r=(\prod_{i=1}^{r-1}x_i^2)x_r\cdot (\prod_{i=r+1}^{m}x_i^2)x_{m+1}^2$, and
$c_{m+1}=(\prod_{i=1}^{m}x_i^2)x_{m+1}$ for $1\leq p,q,r\leq m$.
Then $V_1\sqcup V_2\sqcup V_3$ is the vertex set of the sorted graph $G_{B_2(u),s}$ and the induced subgraph of $G_{B_2(u),s}$ on $V_i$ is a
complete graph for $i=1,2,3$ and  $\{a_p,b_q\}\in E(G_{B_2(u),s})$ if and only if $p\leq q$,
$\{b_q,c_t\}\in E(G_{B_2(u),s})$ if and only if $q\leq t$.
Hence  $B_2(u)$ is Freiman by Lemmas \ref{chordal} and \ref{judge}.
This proof is completed.
\end{proof}

 \begin{Theorem}\label{main3}
 Let $d\geq 5$ be an integer, and  $u$ be a $2$-bounded monomial of degree $d$ with  $x_1^2\nmid u$. Then $B_2(u)$ is Freiman if and only if $u$ is a $2$-bounded monomial as given
in  Theorem \ref{general}.
\end{Theorem}
\begin{proof}$(\Leftarrow)$  It follows from Theorem \ref{general}.

$(\Rightarrow)$  Let $u_1=x_1(\prod_{i=2}^{m+1}x_i^2)$, $u_2=x_1(\prod_{i=2}^{m}x_i^2)x_{m+1}x_{m+2}$, $u_3=x_1(\prod_{i=2}^{m}x_i^2)x_{m+2}^2$,
$u_4=x_1(\prod_{i=2}^{m}x_i^2)x_{m+1}$, $u_5=x_1(\prod_{i=2}^{m}x_i^2)x_{m+2}$, $u_6=x_1(\prod_{i=2}^{m-1}x_i^2)x_{m}x_{m+1}^2$. Then
$u_1,u_2,u_3$ and $u_4,u_5,u_6$ are the first three $2$-bounded monomials of degree $d$ with $x_1^2\nmid u_i$ in lexicographically order when $d=2m+1$  and $d=2m$, respectively.

It is enough to show the following two  statements:
\begin{itemize}
 	\item[(1)]  If  $d=2m+1$ and  $u$ is any $2$-bounded monomial of degree $d$ other than $u_1,u_2$ and $u_3$, then $B_2(u)$ is not  Freiman;
\item[(2)] If  $d=2m$ and  $u$ is any $2$-bounded monomial of degree $d$ other than $u_4,u_5$ and $u_6$, then $B_2(u)$ is also not  Freiman.
\end{itemize}
 Let $u=x_{i_1}x_{i_2}\cdots x_{i_d}$ with $i_1\leq i_2\leq \cdots\leq i_d$ and  $x_1^2\nmid u$.
  If case (1) happens, then there exists some $i\geq m+3$ such that $x_{i}|u$ or
  $m+1\leq i_{d-2}\leq i_{d-1}\leq i_{d}\leq m+2$.
If there exists some $i\geq m+3$ such that $x_{i}|u$, then    $\{v_1,v_2,v_3,v_4\}\subseteq G(B_2(u))$, where
$v_1=(\prod\limits_{i=1}^{m-1}x_i^2)x_{m}x_{m+1}^2$, $v_2=x_1(\prod\limits_{i=2}^{m}x_i^2)x_{m+1}x_{m+2}$, $v_3=(\prod\limits_{i=1}^{m}x_i^2)x_{m+2}$ and $v_4=(\prod\limits_{i=1}^{m-1}x_i^2)x_{m}x_{m+1}x_{m+3}$.
It is easy to check that the sorted graph $G_{B_2(u),s}$ contains an induced 4-cycle with vertices $v_1,v_2,v_3$ and $v_4$. It follows that $B_2(u)$ is not Freiman by  Lemma \ref{judge}.

If $m+1\leq i_{d-2}\leq i_{d-1}\leq i_{d}\leq m+2$, then $G(B_2(v))\subseteq G(B_2(u))$, where $v=(\prod_{i=1}^{m-2}x_i^2)x_{m-1}x_mx_{m+1}^2x_{m+2}$ and the sorted graph $G_{B_2(v),s}$ is  an
induced subgraph of the sorted graph $G_{B_2(u),s}$. By repeatedly applying  Lemma \ref{isomorphic}, we obtain that  $B_2(v)$ is Freiman if and only if $B_2(w)$ is Freiman, where $w=x_1x_2x_3^2x_4$.
However,  the sorted graph $G_{B_2(w),s}$ contains an induced $4$-cycle with the vertex set
$\{x_1^2x_2^2x_3,x_1x_2^2x_3^2,x_1x_2x_3^2x_4,x_1^2x_2x_3x_4\}$. Hence $B_2(u)$ is not Freiman by  Lemma \ref{judge}.

 If case (2) happens, then there exists some $i\geq m+3$ such that $x_{i}|u$, or
 $i_{d}= m+2$ and  $i_{d-1}\geq m+1$, or $i_{d}=i_{d-1}= m+1$ and
  $i_{d-2}=i_{d-3}=m$.

If there exists some $i\geq m+3$ such that $x_{i}|u$, or  $i_{d}= m+2$ and  $i_{d-1}\geq m+1$, then $B_2(u)$ satisfies $B_2(u)\supseteq B_2(x_1(\prod\limits_{i=2}^{m}x_i^2)x_{m+3})\supseteq (\prod_{i=1}^{m}x_i)B_{1}((\prod\limits_{i=2}^{m}x_i)x_{m+3})$ or $B_2(u)\supseteq B_2(x_1(\prod\limits_{i=2}^{m-1}x_i^2)x_mx_{m+1}x_{m+2})\supseteq (\prod_{i=1}^mx_i)B_{1}((\prod\limits_{i=2}^{m-1}x_i)x_{m+1}x_{m+2})$ respectively. Moreover, the sorted graphs of $B_{1}((\prod\limits_{i=2}^{m}x_i)x_{m+3})$  and $B_{1}((\prod\limits_{i=2}^{m-1}x_i)x_{m+1}x_{m+2})$ are isomorphic to the induced subgraphs of  $G_{B_2(u),s}$, respectively.
In both cases, we obtain that  $B_2(u)$ is not Freiman by \cite[Theorem 2.5 (1)]{ZZC2}  and Lemma \ref{judge}.

Under the condition that $i_{d}=i_{d-1}= m+1$ and $i_{d-2}=i_{d-3}=m$. If $m=3$, then $B_2(u)\supseteq B_2(x_1x_2x_3^2x_4^2)\supseteq x_1x_2x_3x_4B_{1}(x_3x_4)$ and  the sorted graph
$G_{B_1(x_3x_4),s}$ is isomorphic to an induced subgraph of the  sorted graph  $G_{B_2(u),s}$. If $m\geq 4$, then
$B_2(u)\supseteq B_2(x_1(\prod\limits_{i=2}^{m-2}x_i^2)x_{m-1}x_{m}^2x_{m+1}^2)\supseteq (\prod_{i=1}^{m+1}x_i)B_{1}((\prod\limits_{i=2}^{m-2}x_i)x_mrx_{m+1})$.
Similarly, the sorted graphs of $B_{1}((\prod\limits_{i=2}^{m-2}x_i)x_mx_{m+1})$ is also isomorphic to an induced subgraph of the  sorted graph  $G_{B_2(u),s}$. Hence we obtain that
$B_2(u)$ is not Freiman by \cite[Theorem 2.5 (1)]{ZZC2}  and Lemma \ref{judge}.
\end{proof}

\medskip
\hspace{-6mm} {\bf Acknowledgments}

 \vspace{3mm}
\hspace{-6mm} The authors gratefully acknowledge the use of the computer
algebra system CoCoA \cite{Co} for our experiments.  This research is supported by the National Natural Science Foundation of China (No.11271275) and  by foundation of the Priority Academic Program Development of Jiangsu Higher Education Institutions.

\medskip

\end{document}